\documentclass[11pt]{article}
\usepackage[margin=1.25in]{geometry}

\usepackage{url,multicol,fancyhdr}
\usepackage[disable]{todonotes}

\usepackage[bottom]{footmisc} 

\usepackage[pdfencoding=auto, psdextra]{hyperref} 
\hypersetup{
    linktocpage,
    colorlinks = True,
    citecolor=gray,
    filecolor=black,
    linkcolor=blue,
    urlcolor=blue
}
\usepackage[alphabetic]{amsrefs}
\usepackage{amsthm,stmaryrd,amsfonts,amssymb,mathtools} 
\usepackage{cleveref}
\usepackage{bookmark}

\usepackage{tikz}
\usepackage{tikz-cd}

\usepackage{bbm} 

\usepackage{accents} 

\theoremstyle{plain}
\newtheorem{theorem}{Theorem}[section]
\newtheorem{proposition}[theorem]{Proposition}
\newtheorem{corollary}[theorem]{Corollary}
\newtheorem{lemma}[theorem]{Lemma}
\newtheorem{remark}[theorem]{Remark}

\theoremstyle{definition}
\newtheorem{definition}[theorem]{Definition}
\newtheorem{example}[theorem]{Example}

\crefname{definition}{Definition}{Definitions}
\crefformat{section}{\S#2#1#3} 

    \newcommand{\fm}[1]{\textbf{#1}} 
    \renewcommand{\sf}[1]{\textsf{#1}}

    \newcommand{\cat}[1]{\mathbf{#1}}

    \newcommand{\Set}{\cat{Set}}
    
    \newcommand{\Cat}{\cat{Cat}}

    \newcommand{\Aut}{\cat{Aut}}

    \newcommand{\Comon}{\cat{Comon}}
    
    \newcommand{\Comod}{\cat{Comod}}

    \newcommand{\op}{\sf{op}}

    \newcommand{\var}[1]{\mathcal{#1}}
    \newcommand{\A}{\var{A}}
    \newcommand{\B}{\var{B}}
    \newcommand{\D}{\var{D}}

    \newcommand{\V}{\var{V}}

    \newcommand{\1}{\mathbf{1}}
    \newcommand{\II}{\mathbbm{1}}

    \tikzset{
        symbol/.style={%
            draw=none,
            every to/.append style={%
                edge node={node [sloped, allow upside down, auto=false]{$#1$}}}
        }
    }

    \newcommand{\xto}[1]{\xrightarrow{#1}}

    \newcommand{\braid}{\mathbbm{x}}
    \renewcommand{\cot}[1]{\underset{#1}{\diamond}}

\begin{document}
\title{Smash Products for Non-cartesian Internal Prestacks}
\author{Liang Ze Wong}
\date{}

\maketitle

\begin{abstract}
	The smash product construction (or the Grothendieck construction) takes a functor (or prestack) $F \colon B^\op \to \Cat$ and returns a fibration $p \colon A \to B$.
	In this paper, we develop an analogue of the smash product for prestacks internal to a non-cartesian monoidal category.
	Our construction simultaneously generalizes the Grothendieck construction for prestacks and smash products for $B$-module algebras over a bialgebra $B$.
	Further, taking fibers or coinvariants allows one to recover the original prestack.
\end{abstract}

\tableofcontents

\section{Introduction}
	Given a group $G$ acting on another group $A$ via a homomorphism $\varphi \colon G \to \Aut(A)$, we may form the semi-direct product $A \rtimes_\varphi G$, or simply $A \rtimes G$.
	There is also a projection $\pi \colon A \rtimes G \twoheadrightarrow G$, and taking the kernel of $\pi$ allows us to recover $A$.
	This paper synthesizes two classical generalizations of the semi-direct product.

	The first is the \emph{Grothendieck construction} \cite{grothendieck1961categories}.
	Instead of a group $G$ acting on another group $N$, we now have a category $\B$ acting on a family of other categories $\{\A_b\}_{b \in \B}$ via a functor $\varphi \colon \B^\op \to \Cat$ sending $b$ to $\A_b$. 
	Such functors are also known as \emph{(split) prestacks}.
	The Grothendieck construction then takes a split prestack and returns a \emph{fibration} $\pi \colon \A \rtimes \B \to \B$ whose \emph{fibers} allow us to recover the categories $\A_b$ that we started with.

	The second generalization is the \emph{smash product} construction \cite{cohen1984group}.
	This time, instead of a group acting on another group, we start with a group $G$ acting on a $k$-algebra $A$.
	We may then form the smash product $A \rtimes G$ (or $A \# G$), which is another $k$-algebra.
	Instead of an algebra homomorphism $A \rtimes G \to G$,  we have a $G$-grading on $A \rtimes G$ whose identity component is the original algebra $A$; equivalently, we have a $kG$-comodule algebra $A \rtimes G$ whose \emph{coinvariant} subalgebra is $A$.
	More generally, given a Hopf algebra $H$ acting on another algebra $A$ (i.e.\ a $H$-module algebra), we may form the smash product $A \rtimes H$ which is a $H$-comodule algebra, and taking the coinvariant subalgebra of $A \rtimes H$ allows us to recover $A$ \cite{blattner1985duality,van1984duality}.
	Although the antipode of the Hopf algebra $H$ is used in the definition of the smash product, it is not actually \emph{required}: we may in fact form the smash product $A \rtimes B$ for a bialgebra $B$ acting on $A$, which coincides with the usual smash product if $B$ is a Hopf algebra.

	The starting point of this paper is the observation that categories $\B$ and bialgebras $B$ are both examples of \emph{internal categories} \cite{aguiar1997internal}.
	In fact, they are \emph{comonoidal} internal categories (which we define in \Cref{sec:comonoidal-internal-cat}), and we may thus define \emph{comodule categories} and \emph{prestacks} over them (\Cref{sec:internal-prestacks}).
	In \Cref{sec:smash-products}, we define \emph{smash products} of prestacks, and in \Cref{sec:coinvariants} we show that taking coinvariants allows us to recover the original prestack.
	Some necessary lemmas regarding comonoids and comodules will be provided in \Cref{sec:comonoids-comodules}.

	The reader might find many of the statements and proofs in this paper rather technical and unmotivated.
	This is because they were developed in the following manner:
	\begin{enumerate}
		\item Identify a notion for ordinary categories (i.e.\ categories internal to $\Set$);
		\item Define this notion for categories internal to an arbitrary monoidal category $\V$, in the language of comonoids and comodules;
		\item Prove the necessary statements using \emph{string diagrams};
		\item Transfer this proof into commutative diagrams.
	\end{enumerate}
	Consequently, the results and proofs that end up in this paper are already one step removed from the original \emph{method} of proof (string diagrams), and three steps removed from the original \emph{motivation} (ordinary category theory)! 
	Future versions of this paper might attempt to better motivate the results, and present them using string diagrams.
	For now, we encourage the reader to keep the original categorical constructions in mind and work out the statements and proofs for themselves in string diagrams.

\section{Comonoids and comodules} \label{sec:comonoids-comodules}
	In this section, we give a quick overview of comonoids and comodules.
	Throughout, we assume that $(\V, \otimes, \1, \braid)$ is a symmetric monoidal category, where $\braid$ denotes the symmetry.
	We will further assume that $\V$ is \emph{regular} in the following sense:

	\begin{definition}[\cite{aguiar1997internal}*{Definition 2.1.1}] \label{def:regular}
		A monoidal category $(\V,\otimes, \1)$ is \fm{regular} if it has all equalizers, and $\otimes$ preserves them (in both variables).
        In other words, if $\begin{tikzcd} E \ar[r, rightarrowtail, "\mathsf{eq}"] & X \end{tikzcd}$ is the equalizer of
        $
            \begin{tikzcd}
                X \ar[r, "f", shift left] \ar[r, "g"', shift right] & Y,
            \end{tikzcd}
        $
        then $A \otimes E \otimes B \xto{\; A \otimes \mathsf{eq} \otimes B \;} A \otimes X \otimes B$ is the equalizer of 
        \[
            \begin{tikzcd}[column sep = large]
                A \otimes X \otimes B \ar[r, "A \otimes f \otimes B" shift left] \ar[r, "A \otimes g \otimes B"', shift right] & A \otimes Y \otimes B.
            \end{tikzcd}
        \]
	\end{definition}

	For $C,D$ comonoids in $\V$, let $_C\Comod_D$ denote the category of left $C$-, right $D$-bicomodules, or $(C,D)$-comodules.
	When either $C$ or $D$ is the monoidal unit $\1$, we write $_C\Comod := {}_C\Comod_\1$ and $\Comod_D := {}_1\Comod_D$.
	The maps in $_C\Comod_D$ are comodule maps respecting both the $C$ and $D$ coactions.
	More generally, we have:

	\begin{definition}
		Let $f \colon C \to D$ be a comonoid map, $M \in \Comod_C$ and $N \in \Comod_D$.

		A \fm{(comodule) map over $f$} is a map $\varphi \colon M \to N$ such that the diagram on the left commutes, where $\rho$ denotes the respective right coactions.
		\[
			\begin{tikzcd}[row sep = large]
				M \ar[r, "\varphi"] \ar[d, "\rho"'] & N \ar[d, "\rho"]
				\\
				M \otimes C \ar[r, "\varphi \otimes f"] & N \otimes D
			\end{tikzcd}
			\quad \quad \quad \quad  \quad
			\begin{tikzcd}[sep = large]
				M \ar[r,"\varphi"] \ar[d, "\rho"', dotted] & N \ar[d, "\rho", dotted]
				\\
				C \ar[r, "f"] & D
			\end{tikzcd}			
		\]
		We use the diagram on the right as an abbreviation of the diagram on the left.	
		In particular, the dotted arrows indicate that $M$ has a $C$-coaction and $N$ has a $D$-coaction.
		In the special case where $C = D$ and $f = 1_C$, we say that $\varphi$ is a \fm{map over $C$}.
		We may similarly define maps over $f$ for left comodules.
		For bicomodules, we may define maps over $(f,g)$, or simply maps over $f$ if $g = f$.
		Thus, maps in $\Comod_C$, $_{C}\Comod$ and $_C \Comod_C$ are maps over $C$. 
	\end{definition}

	\begin{lemma}
		Let $f \colon C \to D$ be a comonoid map, $M \in \Comod_C$ and $N \in \Comod_D$.	

		A map $\varphi \colon M \to N$ over $f \colon C \to D$ is equivalently a $D$-comodule map $f_* M \to N$, where $f_*$ is the corestriction along $f$.
	\end{lemma}

    \begin{definition}
        Let $B,C,D$ be comonoids, and let $M \in {}_B \Comod_C$ and $N \in {}_C \Comod_D$.
        The \fm{cotensor over $C$} of $M$ and $N$ is the equalizer:
        \[
            \begin{tikzcd}[column sep = large]
                M \cot{C} N \ar[r, rightarrowtail] &
                M \otimes N \ar[r, shift left, "\rho_M \otimes N"] \ar[r, shift right, "M \otimes \lambda_N"'] &
                M \otimes C \otimes N
            \end{tikzcd}
        \]
    \end{definition}        

    \begin{proposition}[\cite{aguiar1997internal}*{Proposition 2.2.1}] \label{prop:cotensor-coact}
        When $\V$ is a regular, $M \cot{C} N$ has a right $D$-coaction induced by the coaction on $N$: 
        \[
            \begin{tikzcd}[sep = large]
                M \cot{C} N \ar[r, rightarrowtail] \ar[d, dashed, "M \boxtimes \rho_N"'] & M \otimes N \ar[r, shift left] \ar[r, shift right] \ar[d, "M \otimes \rho_N"'] & M \otimes C \otimes N \ar[d, "M \otimes C \otimes \rho_N"]
                \\
                M \cot{C} N \otimes D \ar[r, rightarrowtail] & M \otimes N \otimes D \ar[r, shift right] \ar[r, shift left] & M \otimes C \otimes N \otimes D
            \end{tikzcd}
        \]
        Similarly, $M \cot{C} N$ has a left $B$-coaction making $M \cot{C} N$ an object of $_B\Comod_D$.
    \end{proposition}

    \begin{lemma}[\cite{aguiar1997internal}*{Lemma 7.1.1}] \label{lem:cot-iso}
    	Let $M \in {}_B \Comod_C, N \in {}_C\Comod_D, M' \in {}_{B'}\Comod_{C'}$ and $N' \in {}_{C'}\Comod_{D'}$.
    	\[
    		\begin{tikzcd}[sep = small]
    			 & M \ar[dr, dotted] \ar[dl, dotted] &  & N \ar[dr, dotted] \ar[dl, dotted] & 
    			\\	
    			B &  & C &  & D
    		\end{tikzcd}
    		\quad \quad \quad
    		\begin{tikzcd}[sep = small]
    			 & M' \ar[dr, dotted] \ar[dl, dotted] &  & N' \ar[dr, dotted] \ar[dl, dotted] & 
    			\\
    			B' &  & C' &  & D'
    		\end{tikzcd}    		
    	\]
    	Then there is a canonical isomorphism in $_{B\otimes B'}\Comod_{D \otimes D'}$
    	\[
    		(M\cot{C} N) \otimes (M' \cot{C'} N') \cong (M \otimes M') \cot{C \otimes C'} (N \otimes N')
    	\]
    	natural in $M, N, M'$ and $N'$.
    \end{lemma}

    It is further shown in \cite{aguiar1997internal}*{\S 2.2} that cotensoring extends to a functor 
    \[
    	-\cot{C}-\colon	_B \Comod_C \times {}_C \Comod_D \to {}_B \Comod_D.
    \]
    In particular, if $\varphi \colon M \to M'$ and $\psi \colon N \to N'$ are maps in $_B\Comod_C$ and $_C \Comod_D$, there is a $_B \Comod_D$-map $\varphi \cot{C} \psi \colon M \cot{C}N \to M'\cot{C}N'$.
    More generally, we have:

    \begin{proposition} \label{prop:cotensor-maps}
    	Let $f\colon B \to B', g \colon C \to C'$ and $h \colon D \to D'$ be comonoid maps, 
    	and let $M \in {}_B \Comod_C, N \in {}_C\Comod_D, M' \in {}_{B'}\Comod_{C'}$ and $N' \in {}_{C'}\Comod_{D'}$, 
    	and suppose we have $\varphi \colon M \to M'$ over $(f,g)$ and $\psi \colon N \to N'$ over $(g,h)$.
    	\[
    		\begin{tikzcd}[row sep = small]
    			 & M \ar[dr, dotted] \ar[dl, dotted] \ar[dd, "\varphi"'] &  & N \ar[dl, dotted] \ar[dr, dotted] \ar[dd, "\psi"] & 
    			\\
    			B \ar[dd, "f"'] &  & C \ar[dd, "g"] &  & D \ar[dd, "h"]
    			\\
    			 & M' \ar[dr, dotted] \ar[dl, dotted] &  & N' \ar[dl, dotted] \ar[dr, dotted] & 
    			\\
    			B' &  & C' &  & D'   			
    		\end{tikzcd}
    	\]
    	Then there is a map $\varphi \cot{g} \psi \colon M \cot{C} N \to M' \cot{C'} N'$ over $(f,h)$.
    \end{proposition}
    \begin{proof}
    	The map $\varphi\cot{g} \psi$  is induced by:
        \[
            \begin{tikzcd}[sep = large]
                M \cot{C} N \ar[r, rightarrowtail] \ar[d, dashed, "\varphi \cot{g} \psi"'] & M \otimes N \ar[r, shift left] \ar[r, shift right] \ar[d, "\varphi \otimes \psi"'] & M \otimes C \otimes N \ar[d, "\varphi \otimes g \otimes \psi"]
                \\
                M' \cot{C'} N' \ar[r, rightarrowtail] & M' \otimes N' \ar[r, shift right] \ar[r, shift left] & M' \otimes C' \otimes N'
            \end{tikzcd}
        \]    	
        This is a comodule map over $h$ if the left-most face of the following diagram commutes,
        \[
        	\begin{tikzcd}[sep = normal]
        		M_{C}N \ar[rr, rightarrowtail] \ar[dd, dashed] \ar[dr, dashed]  & & M N \ar[rr, shift right] \ar[rr, shift left] \ar[dd] \ar[dr]  & & M  C  N \ar[dd] \ar[dr]
        		\\
        		& M'_{C'}N'\ar[rr, rightarrowtail, crossing over]\ar[dd, dashed, crossing over] & & M'  N' \ar[rr, shift right, crossing over] \ar[rr, shift left, crossing over] \ar[dd, crossing over]  & & M'  C'  N' \ar[dd]
        		\\
        		M_{C}N  D \ar[rr, rightarrowtail]\ar[dr, dashed] & & M N  D \ar[rr, shift right] \ar[rr, shift left] \ar[dr]  & & M  C  N  D \ar[dr]
        		\\
        		& M'_{C'}N'  D'\ar[rr, rightarrowtail] & & M'  N'  D' \ar[rr, shift right] \ar[rr, shift left]  & & M'  C'  N'  D'
        	\end{tikzcd}
        \]
        where we have omitted $\otimes$ and $\cot{}$ for brevity.
        But both composites that make up the left-most face are maps uniquely induced by the diagonal map $M \cot{C} N \to M' \otimes N' \otimes D'$, hence are equal.
        Similarly, $\varphi\cot{g} \psi$ is a comodule map over $f$.
    \end{proof}

    \begin{theorem}[\cite{aguiar1997internal}*{Theorem 2.2.1}]
    	There is a bicategory whose objects are comonoids in $\V$, and whose category of arrows from $C$ to $D$ is $_C\Comod_D$.
    \end{theorem}

    \begin{corollary}
    	For $C$ a comonoid in $\V$, $(_C\Comod_C, \cot{C}, C)$ is a monoidal category.
    \end{corollary}



    We conclude this section with some useful lemmas.

    \begin{lemma}\label{lem:cot-comonoid}
    	Let $(D,d,e), (M_1, \delta_1, \epsilon_1)$ and $(M_2, \delta_2, \epsilon_2)$ be comonoids in $\V$.
    	If each $M_i$ is in ${}_D \Comod_D$, and $\delta_i$ and $\epsilon_i$ are maps over $d$ and $e$, then $M_1 \cot{D} M_2$ is also a comonoid.
    \end{lemma}
    \begin{proof}
    	By \Cref{prop:cotensor-maps}, since each $\delta_i$ is a map over $d$, we have a map $\delta_1 \cot{d} \delta_2$, which we may compose with the isomorphism from \Cref{lem:cot-iso} to obtain a comultiplication:
    	\[
    		\begin{tikzcd}
    			M_1 \cot{D} M_2 \ar[r, "\delta_1 \cot{d} \delta_2"] & (M_1 \otimes M_1) \cot{D \otimes D} (M_2 \otimes M_2) \ar[r, "\cong"] & (M_1 \cot{D} M_2) \otimes (M_1 \cot{D} M_2)
    		\end{tikzcd}
    	\]
    	Similarly, since each $\epsilon_i$ is a comodule map over $e$, we have a counit
    	\[
    		\begin{tikzcd}
    			M_1 \cot{D} M_2 \ar[r, "\epsilon_1 \cot{e} \epsilon_2"] & \1 \cot{\1} \1 \cong \1.
    		\end{tikzcd}
    	\]
    	The reader may verify that these maps make $M_1 \cot{D} M_2$ a comonoid.
    \end{proof}

    \begin{lemma} \label{lem:coaction-is-comonoid-map}
    	Let $(C, \delta, \epsilon)$ and $(D, d, e)$ be comonoids, and suppose that $C$ is a $D$-comodule with coaction $p \colon C \to C \otimes D$.
    	Then $p$ is a comonoid map if and only if $\delta$ is a map over $d$:
    	\[
    		\begin{tikzcd}[sep = large]
    			C \ar[r, "\delta"] \ar[d, "p"', dotted] & C \otimes C \ar[d, "p \otimes p", dotted]
    			\\
    			D \ar[r, "d"] & D \otimes D
    		\end{tikzcd}
    	\]
    \end{lemma}
    \begin{proof}
    	Note that $p$ always preserves counits, so $p$ is a comonoid map if and only if it also preserves comultiplication.

    	The diagram in the lemma commutes precisely when the left pentagon in the following diagram commutes:
    	\[
    		\begin{tikzcd}
    			C \ar[rr, "p"] \ar[dd, "\delta"'] & & C \otimes D \ar[d, "d \otimes \delta"]
    			\\
    			& & C \otimes C \otimes D \otimes D \ar[d, "C \otimes \braid \otimes D"]
    			\\
    			C \otimes C \ar[r, "p \otimes p"] & C \otimes D \otimes C \otimes D \ar[ur, "C \otimes \braid \otimes D"] \ar[r, equals] & C \otimes D \otimes C \otimes D
    		\end{tikzcd}
    	\]
    	The outer square then says that $p$ is a comonoid map\footnote{Note that we need $\V$ to be \emph{symmetric}, not just braided, for the bottom-right corner to commute!}.  

    	Conversely, if $p$ is a comonoid map, the left pentagon in the following diagram commutes:
    	\[
    		\begin{tikzcd}
    			C \ar[dd, "p"'] \ar[rr, "\delta"] & & C \otimes C \ar[d, "p \otimes p"]
    			\\
    			& & C \otimes D \otimes C \otimes D \ar[d, "C \otimes \braid \otimes D"]
    			\\
    			C \otimes D \ar[r, "\delta \otimes d"] & C \otimes C \otimes D \otimes D \ar[ur, "C \otimes \braid \otimes D"] \ar[r, equals] & C \otimes C \otimes D \otimes D
    		\end{tikzcd}
    	\]    	
    	The outer square then says that $\delta$ is a map over $d$.
    \end{proof}

    \begin{lemma} \label{lem:coaction-induced-by-comonoid-map}
    	Let $C, D$ be comonoids, and $p \colon C \to C \otimes D$ be a $D$-coaction that is also a comonoid map.
    	Then $p$ is induced by a comonoid map $q \colon C \to D$.
    \end{lemma}
    \begin{proof}
    	The counit $e\colon C \to \1$ is a comonoid map, so the composite
    	\[
    		\begin{tikzcd}
    			q \colon C \ar[r, "p"] & C \otimes D \ar[r, "e \otimes D"] & D
    		\end{tikzcd}
    	\]
    	is a comonoid map.
    	The left square of the following diagram commutes because $p$ is a comonoid map; the upper-right square commutes because $p$ is a coaction.
    	\[
    		\begin{tikzcd}[column sep = large]
    			C \ar[r, "p"] \ar[dd, "d"'] & C \otimes D \ar[r, equals] \ar[d, "d \otimes d"'] & C \otimes D \ar[d, equals]
    			\\
    			& C \otimes C \otimes D \otimes D \ar[r, "C \otimes e \otimes e \otimes D"] \ar[d, "C \otimes \braid \otimes D"'] & C \otimes D \ar[d, equals]
    			\\
    			C \otimes C \ar[r, "p \otimes p"] \ar[rr, bend right = 10, "C \otimes q"'] & C \otimes D \otimes C \otimes D \ar[r, "C \otimes e \otimes e \otimes D"] & C \otimes D
    		\end{tikzcd}
    	\]
    	The outer diagram then says that $q$ induces $p$.
    \end{proof}    

    \begin{remark}
    	The converse of Lemma \ref{lem:coaction-induced-by-comonoid-map} does not hold: given an arbitrary comonoid map $q \colon C \to D$, the coaction
    	\[
    		\begin{tikzcd}
    			C \ar[r, "\delta"] & C \otimes C \ar[r, "C \otimes q"] & C \otimes D
    		\end{tikzcd}
    	\]
    	need not be a comonoid map, because $\delta \colon C \to C \otimes C$ is not a comonoid map (unless $C$ is cocommutative).
    	Thus the two equivalent conditions in Lemma \ref{lem:coaction-is-comonoid-map} are stronger than the condition in Lemma \ref{lem:coaction-induced-by-comonoid-map}.
    \end{remark}

    \begin{remark}
    	Note that for any comonoid map $q \colon C \to D$ inducing a coaction $p \colon C \to C \otimes D$, the following diagram always commutes:
    	\[
    		\begin{tikzcd}
    			C \ar[rr, "q"] \ar[dr, "p"', dotted] & & D \ar[dl, "d", dotted]
    			\\
    			& D
    		\end{tikzcd}
    	\]
    \end{remark}


    \begin{lemma} \label{lem:coaction-over-cocomm}
    	Let $C$ be a cocommutative comonoid, and $M \in \Comod_C$ with coaction $\rho \colon M \to M \otimes C$.
    	Then $\rho$ is a  map over $\delta$:
    	\[
    		\begin{tikzcd}[sep = large]
    			M \ar[r, "\rho"] \ar[d, "\rho"', dotted]  & M \otimes C \ar[d, "\rho \otimes \delta" , dotted]
    			\\
    			C \ar[r, "\delta"] & C \otimes C
    		\end{tikzcd}
    	\]
    \end{lemma}
    \begin{proof}
    	We need the following diagram to commute:
    	\[
    		\begin{tikzcd}[column sep = normal]
    			M \ar[r, "\rho"] \ar[d,"\rho"'] & M \otimes C \ar[r, "\rho \otimes C"] & M \otimes C \otimes C \ar[d, "M \otimes C \otimes \delta"]
    			\\
    			M \otimes C \ar[d, "\rho\otimes C"'] & &  M \otimes C \otimes C \otimes C \ar[d, "M \otimes \braid \otimes C"]
    			\\
    			M \otimes C \otimes C \ar[rr, "M \otimes C \otimes \delta"] & & M \otimes C \otimes C \otimes C
    		\end{tikzcd}    		
    	\]
    	Since $(\rho\otimes C) \rho = (M \otimes \delta)\rho$, this is equivalent to the following diagram commuting,
    	\[
    		\begin{tikzcd}[sep = large]
    			M \ar[r, "\rho"] \ar[d,"\rho"'] & M \otimes C \ar[r, "M \otimes \delta"] \ar[d, equals] & M \otimes C \otimes C \ar[d, "M \otimes C \otimes \delta"]
    			\\
    			M \otimes C \ar[d, "M \otimes \delta"'] \ar[r, equals] &  M \otimes C \ar[r, "M \otimes (\delta \otimes C) \delta"] \ar[d, "M \otimes (\delta \otimes C) \delta"] &  M \otimes C \otimes C \otimes C \ar[d, "M \otimes \braid \otimes C"]
    			\\
    			M \otimes C \otimes C \ar[r, "M \otimes C \otimes \delta"] & M \otimes C \otimes C \otimes C \ar[r, equals] & M \otimes C \otimes C \otimes C
    		\end{tikzcd} 
       	\]
       	whose bottom-right square commutes because $C$ is cocommutative.
    \end{proof}

\section{Comonoidal internal categories} \label{sec:comonoidal-internal-cat}

	We take our definition of a category internal to a regular monoidal category $(\V, \otimes, \1)$ from \cite{aguiar1997internal}.
	Importantly, $\V$ is not required to be cartesian i.e.\ the monoidal product $\otimes$ is not necessarily the cartesian product $\times$.

    \begin{definition}[\cite{aguiar1997internal}*{Definition 2.3.1}]
        A $\V$-\fm{internal category} consists of a comonoid $C$ in $\V$ and a monoid $A$ in $_C \Comod_C$.

        In detail, an internal category is a tuple $\A = (C, A, d, e, \sigma, \tau, u, m)$ with
        \begin{enumerate}
            \item a \emph{comonoid of objects} $C \in \Comon(\V)$, with comultiplication $d \colon C \to C \otimes C$ and counit $e \colon C \to \1$;
            \item a \emph{comodule of maps} $A \in {}_{C}\Comod_{C}$, with coactions\footnote{$\sigma$ for `source' and $\tau$ for `target'.} $\sigma \colon A \to C \otimes A$ and $\tau \colon A \to A \otimes C$;
            \item and \emph{identity} and \emph{composition} comodule maps
            
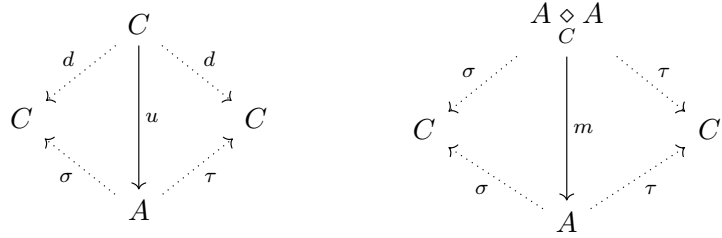
\[
            	\begin{aligned}[b]
            	\begin{tikzcd}
            		& C \ar[dl, dotted, "d"'] \ar[dr, dotted, "d"] \ar[dd, "u"] &
            		\\
            		C & & C
            		\\
            		& A \ar[ul, dotted, "\sigma"]  \ar[ur, dotted, "\tau"'] &
            	\end{tikzcd}
            	\end{aligned}
                \quad\quad\quad\quad
                \begin{aligned}[b]
            	\begin{tikzcd}
            		& A \cot{C} A \ar[dl, dotted, "\sigma"'] \ar[dr, dotted, "\tau"] \ar[dd, "m"] &
            		\\
            		C & & C
            		\\
            		& A \ar[ul, dotted, "\sigma"]  \ar[ur, dotted, "\tau"'] &
            	\end{tikzcd}    
            	\end{aligned}    	
            \]
            satisfying associativity and unitality.
        \end{enumerate}
        For brevity, we will sometimes refer to an internal category $\A$ using subtuples such as $(C, A)$.
    \end{definition}

    \begin{remark}
    	The definition of an internal category does not require the comonoid of objects $C$ to be cocommutative.
    	However, it does not seem possible to define internal prestacks or the internal Grothendieck construction without cocommutativity of objects.
    	The internal categories that we subsequently consider will all have cocommutative comonoids of objects.
    	In such a situation, the left coaction $\sigma$ induces a right coaction $\braid \sigma$.
    	Similarly, the right coaction $\tau$ induces a left coaction $\braid\tau$.
    \end{remark}

    \begin{example}
    	Any monoid $A$ in $\V$ gives rise to the `one-object' internal category $(\1, A)$.
    	Any comonoid $C$ in $\V$ gives rise to the `discrete' internal category $(C,C)$.
    	(These are denoted $\hat{A}$ and $\underaccent{\check}{C}$ in \cite{aguiar1997internal}*{Example 2.4.1}.)
    \end{example}

    \begin{definition}[\cite{aguiar1997internal}*{Definition 4.1.1}]
        Let $\A = (C, A)$ and $\B = (D, B)$ be internal categories in $\V$.
        An \fm{internal functor} from $\A$ to $\B$ is a tuple $(f, \varphi)$ where
        $f \colon C \to D$ is a comonoid map and $\varphi \colon A \to B$ is a map 
        such that the following diagrams commute:
        \[
        	\begin{aligned}[b]
    		\begin{tikzcd}[sep = large]
        		C \ar[d, "f"'] & \ar[l, "\sigma"', dotted] A \ar[d, "\varphi"] \ar[r, "\tau", dotted] & C \ar[d, "f"]
        		\\
        		D & \ar[l, "\sigma"', dotted] B \ar[r, "\tau", dotted] & D
        	\end{tikzcd}
        	\end{aligned}
        	\quad \quad \quad
            \begin{aligned}[b]
            \begin{tikzcd}[sep = large]
                C \ar[r, "f"]  \ar[d, "u"']& D \ar[d, "u"]
                \\
                A \ar[r, "\varphi"] & B
            \end{tikzcd}
            \end{aligned}
            \quad\quad \quad
            \begin{aligned}[b]
            \begin{tikzcd}[column sep = large]
                A \cot{C} A \ar[r, "\varphi \cot{f} \varphi"] \ar[d, "m"'] & B \cot{D} B \ar[d, "m"]
                \\
                A \ar[r, "\varphi"] & B
            \end{tikzcd}
            \end{aligned}
        \]
    \end{definition}

    \begin{definition}
        Let $\Cat(\V)$ denote the category of internal categories and functors. 
    \end{definition}        

    \begin{remark}
    	It is also possible to define internal transformations between internal functors, making $\Cat(\V)$ a $2$-category, but we will not need the $2$-category structure in this paper.
    \end{remark}

    Recall that if $\V$ is a symmetric monoidal category, its category of comonoids $\Comon(\V)$ is also symmetric monoidal, with the same braiding and monoidal product.
    A similar result holds for internal categories.

    \begin{proposition}[\cite{aguiar1997internal}*{\S 7.1}]
    	$\Cat(\V)$ is a monoidal category, with product
    	\[
    		(C, A) \otimes (D, B) := (C \otimes D, A \otimes B)
    	\]
    	and unit $\II := (\1,\1)$.
    \end{proposition}

    \begin{definition}
    	A \fm{comonoidal internal category} $\B = (D,B)$ is a comonoid in $\Cat(\V)$.
    \end{definition}

    \begin{proposition}
    	Let $\B = (D, B, d, e, \sigma, \tau, u, m)$ be a comonoidal internal category. Then:
    	\begin{enumerate}
    		\item $D$ is cocommutative (i.e.\ $d$ and $e$ are comonoid maps);
    		\item $B$ is a comonoid, and $\sigma$ and $\tau$ are comonoid maps (hence are induced by comonoid maps $s \colon B \to D$ and $t \colon B \to D$);
    		\item $\sigma, \tau$ and $\delta$ (the comultiplication of $B$) are maps over $d$:
    		\[
    			\begin{tikzcd}[column sep = huge, row sep = large]
    				D \ar[d, "d"'] \ar[from = r, dotted, "\sigma"'] & B \ar[d, "\delta"] \ar[r, "\tau", dotted] & D \ar[d,"d"]
    				\\
    				D \otimes D \ar[from = r, dotted, "\sigma \otimes \sigma"'] & B \otimes B \ar[r, dotted, "\tau \otimes \tau"] & D \otimes D
    			\end{tikzcd}
    		\]
    		\[
    			\begin{tikzcd}[sep = large]
    				B \ar[r,"\sigma"] \ar[d, "\sigma"', dotted] & D \otimes B \ar[d, "d \otimes \sigma" , dotted]
    				\\
    				D \ar[r, "d"] & D \otimes D
    			\end{tikzcd}
    			\quad \quad \quad
    			\begin{tikzcd}[sep = large]
    				B \ar[r,"\tau"] \ar[d, "\tau"', dotted] & B \otimes D \ar[d, "\tau \otimes d" , dotted]
    				\\
    				D \ar[r, "d"] & D \otimes D
    			\end{tikzcd}			    			
    		\]
    		\item $B \cot{D} B$ is a comonoid, and $u$ and $m$ are comonoid maps.
    	\end{enumerate}
    \end{proposition}
    \begin{proof}
    	Let $(d', \delta) \colon (D, B) \to (D \otimes D, B \otimes B)$ and $(e', \epsilon) \colon (\1, \1) \to (D,B)$ be internal functors making $\B = (D,B)$ comonoidal.
    	Then $(d,e)$ and $(d',e')$ are both counital comonoidal structures on $D$ such that $d'$ is a comonoid map with respect to $d$.
    	By the Eckmann-Hilton argument, we have $e'=e, d'=d,$ and $D$ is cocommutative.

    	The maps $(\delta, \epsilon)$ make $B$ a comonoid, and $\delta$ is a map over $d$ by definition of an internal functor.
    	By \Cref{lem:coaction-is-comonoid-map}, both $\sigma$ and $\tau$ are comonoid maps, and by \Cref{lem:coaction-over-cocomm}, they are comodule maps over $d$.

    	Since $\delta$ and $\epsilon$ are comodule maps over $d$ and $e$, \Cref{lem:cot-comonoid} shows that $B \cot{D} B$ is a comonoid.
		Finally, since $(d, \delta)$ and $(e, \epsilon)$ are internal functors, $\delta$ and $\epsilon$ are required to make the following diagrams commute: 
		\[
			\begin{aligned}[b]
			\begin{tikzcd}[sep = large]
				D \ar[r, "d"] & D \otimes D
				\\
				B \ar[r, "\delta"] \ar[from = u, "u"'] & B \otimes B \ar[from = u, "u \otimes u"]
			\end{tikzcd}
			\end{aligned}
			\quad\quad\quad\quad
			\begin{aligned}[b]
			\begin{tikzcd}[sep = large]
				D \ar[r, "e"] & \1
				\\
				B \ar[from = u, "u"'] \ar[r, "\epsilon"] & \1 \ar[u, equals]
			\end{tikzcd}
			\end{aligned}
			\quad\quad\quad\quad
			\begin{aligned}[b]
			\begin{tikzcd}
				B \cot{D} B \ar[d, "m"] \ar[r, "\epsilon \cot{e} \epsilon"] & \1 \ar[d, equals]
				\\
				B \ar[r, "\epsilon"] & \1
			\end{tikzcd}
			\end{aligned}		
		\]
		\[
			\begin{aligned}
			\begin{tikzcd}[column sep = large]
				B \cot{D} B \ar[r, "\delta \cot{d} \delta"] \ar[d, "m"']  &  (B \cot{D} B) \otimes (B \cot{D} B) \ar[d, "m \otimes m"]
				\\
				B \ar[r, "\delta"] & B \otimes B
			\end{tikzcd}
			\end{aligned}
		\]		    
		But these are precisely the diagrams that make $u$ and $m$ comonoid maps.
    \end{proof}

    \begin{remark}
    	The previous proposition effectively says that a comonoidal internal category is a `category internal to $\Comon(\V)$'.
    	We write the latter statement in quotes because our definition of internal category requires the ambient monoidal category to be regular, which $\Comon(\V)$ need not be.
    \end{remark}

    \begin{corollary} \label{cor:B-delta-same-source}
    	The following diagram commutes:
    	\[
    		\begin{tikzcd}[sep = large]
    			B \ar[r, "\delta"] \ar[d, "\delta"'] & B \otimes B \ar[d, "(\braid \otimes B)(B \otimes \sigma)"]
    			\\
    			B \otimes B \ar[r, "\sigma \otimes B"] & D \otimes B \otimes B
    		\end{tikzcd}
    	\]
    \end{corollary}
    \begin{proof}
    	Follows from $\sigma$ being a map over $d$.
    \end{proof}
    Thus, although the comultiplicands of $B$ need not be the same (i.e.\ $B$ is not cocommutative), their sources are.
    The analogous statement for targets also holds.

    \begin{example}
    	If $B$ is a bimonoid, its one-object category $(\1, B)$ is comonoidal.

    	If $D$ is a cocommutative comonoid, its discrete category $(D,D)$ is comonoidal.
    \end{example}

\section{Internal Prestacks} \label{sec:internal-prestacks}
	
	\begin{definition}
		Let $\B = (D, B)$ be a comonoidal internal category.
    	A \fm{right $\B$-comodule category} is a right $\B$-comodule in $\Cat(\V)$.

    	In detail, this is the data of an internal category $\A = (C,A)$ along with:
    	\begin{enumerate}
    		\item a $D$-coaction $p \colon C \to C \otimes D$ that is also a comonoid map (hence is induced by a comonoid map $q\colon C \to D$);
    		\item a $B$-coaction $\pi \colon A \to A \otimes B$ that is also a map over $p$;
    		\item such that $(p, \pi) \colon \A \to \A \otimes \B$ is an internal functor.
    	\end{enumerate}
    	We henceforth refer to these as simply $\B$-comodule categories or $\B$-comodules.
    \end{definition}

    Recall that if $\B = (D,B)$ is comonoidal, then so is the discrete category $\D = (D,D)$.


    \begin{lemma} \label{lem:B-D-comod}
    	Let $\B = (D,B)$ be a comonoidal internal category and $\D = (D,D)$ its subcategory of objects.
    	Let $\A = (A,C, \sigma, \tau)$ be a $\D$-comodule category with coaction 
    	\[
    		(p\colon C \to C \otimes D , \pi \colon A \to A \otimes D),
    	\]
    	and let $q \colon C \to D$ be the comonoid map that induces $p$.
    	Then:
    	\begin{enumerate}
    		\item $B \cot{D} C$ is a comonoid, with comultiplication $\Delta:= \delta \cot{d_D} d_C$;
    		\item The $D$-coactions $q_* \sigma$ and $q_* \tau$ on $A$ coincide with $\pi$;
    		\item $\sigma$ and $\tau$ are maps over $d$:
    		\[	
    			\begin{tikzcd}[sep = large]
    				C \otimes A \ar[from = r, "\sigma"'] \ar[d, dotted, "p \otimes \pi"']  & A \ar[d, dotted, "\pi"] \ar[r, "\tau"] & A \otimes C \ar[d, dotted, "\pi \otimes p"]
    				\\
    				D \otimes D \ar[from = r, "d"'] &  D \ar[r, "d"] & D \otimes D
    			\end{tikzcd}
    		\]
    		\item The coactions $\sigma$ and $\tau$ induce $B \cot{D} C$-coactions on $B \cot{D} A$;
    	\end{enumerate}
    \end{lemma}
    \begin{proof}
    	By Lemma \ref{lem:coaction-is-comonoid-map}, since $p$ is a comonoid map, the comultiplication $d_C$ is a comodule map over $d_D$, and the counit $e_C$ is a comodule map over $e_D$.
    	By Lemma \ref{lem:cot-comonoid}, $B \cot{D} C$ is a comonoid.

    	By Lemma \ref{lem:coaction-induced-by-comonoid-map}, $p$ induces a comonoid map $q \colon C \to D$.
    	Corestricting along $q$ makes $A$ a $(D,D)$-bicomodule.
    	Since $\pi$ is a map over $p$, the left square in the following diagram commutes:
    	\[
    		\begin{tikzcd}[column sep = large]
    			A \ar[r, "\pi"] \ar[dd, "\sigma"'] & A \otimes D \ar[d, "\sigma \otimes d"] \ar[ddr, bend left, "\braid", "\cong"']
    			\\
    			& C \otimes A \otimes D \otimes D \ar[d, "A \otimes \braid \otimes D"]
    			\\
    			C \otimes A \ar[r, "p \otimes \pi"] \ar[rr, bend right= 10, "q \otimes A"'] & C \otimes D \otimes A \otimes D \ar[r, "e \otimes D \otimes A \otimes e"] & D \otimes A 
    		\end{tikzcd}
    	\]    	
    	The outer diagram then says that $\braid \pi$ and $q_* \sigma$ coincide.
		A similar diagram (with an identity instead of $\braid$) shows that $\pi$ and $q_* \tau$ coincide.	   

		Again, the following diagram commutes, so $\tau$ is a map over $d$:
    	\[
    		\begin{tikzcd}[column sep = large]
    			A \ar[r, "\pi"] \ar[dd, "\tau"'] & A \otimes D \ar[d, "\tau \otimes d"] 
    			\\
    			& A \otimes C \otimes D \otimes D \ar[d, "A \otimes \braid \otimes D"] \ar[dr, bend left = 15, equals]
    			\\
    			A \otimes C \ar[r, "\pi \otimes p"] & A \otimes D \otimes C \otimes D \ar[r, "A \otimes \braid \otimes D"] & A \otimes C \otimes D \otimes D 
    		\end{tikzcd}
    	\]
    	Similarly, $\sigma$ is a map over $d$.	
    	We may thus form the composites,
		\[
			\begin{tikzcd}[row sep = small]
				B \cot{D} A \ar[r, "\delta\cot{d} \tau"] & (B \otimes B) \cot{D \otimes D} (A \otimes C) \cong (B \cot{D} A) \otimes (B \cot{D} C)
				\\
				B \cot{D} A \ar[r, "\delta \cot{d} \sigma"] & (B \otimes B) \cot{D \otimes D} (C \otimes A) \cong (B \cot{D} C) \otimes (B \cot{D } A)
			\end{tikzcd}
		\]
		which are seen to be coactions.
    \end{proof}

    \begin{definition} \label{def:prestack}
    	Let $\B = (D,B)$ be a comonoidal internal category and $\D = (D,D)$ its subcategory of objects.
    	A \fm{prestack over $\B$} (or a \fm{$\B$-module category}) consists of:
    	\begin{enumerate} \setcounter{enumi}{-1}
    		\item An internal category $\A = (C,A)$ with $C$ cocommutative;
    		\item A coaction $(p, \pi) \colon \A \to \A \otimes \D$;
    		\item A comonoid map $f \colon B \cot{D} C \to C$ satisfying:
    		\[
    			\begin{aligned}[b]
    				\begin{tikzcd}
    					B \cot{D} C \ar[r, "f"] \ar[d,"\sigma"',dotted] & C \ar[d, "\braid p", dotted]
    					\\
    					D \ar[r, equals] & \underset{\phantom{C}}{D}
    				\end{tikzcd}
    			\end{aligned}
    			\quad \quad \quad
    			\begin{aligned}[b]
	    			\begin{tikzcd}
	    				D \cot{D} C \ar[d, "\cong"'] \ar[r, "u \cot{D} C"] & B \cot{D} C \ar[d, "f"]
	    				\\
	    				C \ar[r, equals] & \underset{\phantom{C}}{C}
	    			\end{tikzcd}
    			\end{aligned}
    			\quad \quad \quad
    			\begin{aligned}[b]
    				\begin{tikzcd}
    					B \cot{D} B \cot{D} C \ar[r, "B \cot{D} f"] \ar[d, "m \cot{D} C"'] & B \cot{D} C \ar[d, "f"]
    					\\
    					B \cot{D} C \ar[r, "f"] & C
    				\end{tikzcd}
    			\end{aligned}
    		\]
    		\item \label{def:prestack-functorial-in-B} A map $\varphi \colon B \cot{D} A \to A$ satisfying:
    		\[
    			\begin{tikzcd}[column sep = huge]
    				B \cot{D} C \ar[d, "f"'] \ar[from = r, dotted, "\delta \cot{d} \sigma"'] & B \cot{D} A \ar[d, "\varphi"] \ar[r, "\delta \cot{d} \tau", dotted] & B \cot{D} C \ar[d, "f"]
    				\\
    				C \ar[from = r, dotted, "\sigma"'] & A  \ar[r, "\tau", dotted]& C
    			\end{tikzcd}
			\]    		
    		\[
				\begin{tikzcd}
					D \cot{D} A \ar[d, "\cong"'] \ar[r, "u \cot{D} A"] & B \cot{D} A \ar[d, "\varphi"]
					\\
					A \ar[r, equals] & \underset{\phantom{A}}{A}
				\end{tikzcd}
				\quad \quad \quad
				\begin{tikzcd}
					B \cot{D} B \cot{D} A \ar[r, "B \cot{D} \varphi"] \ar[d, "m\cot{D} A"'] & B \cot{D} A \ar[d, "\varphi"]
					\\
					B \cot{D} A \ar[r, "\varphi"] & A
				\end{tikzcd}
			\]
			\item \label{def:prestack-functorial-in-A} $f$ and $\varphi$ further satisfy:
			\[
				\begin{tikzcd}
					B \cot{D} C \ar[d, "f"'] \ar[r, "B\cot{D} e"] & B \cot{D} A \ar[d, "\varphi"]
					\\
					C \ar[r, "e"] & \underset{\phantom{A}}{A}
				\end{tikzcd}
				\quad \quad
				\begin{tikzcd}
					B \cot{D} (A \cot{C} A) \ar[r, "B \cot{D} m"] \ar[d, "\varphi_2"'] & B \cot{D} A \ar[d, "\varphi"]
					\\
					A \cot{C} A \ar[r] & A
				\end{tikzcd}
			\]					
    	\end{enumerate}
    \end{definition}

    The map $\varphi_2$ is given by the following lemma:
	\begin{lemma} 
		There is an action $\varphi_2 \colon B \cot{D} (A \cot{C} A) \to A \cot{C} A$.
	\end{lemma}
	\begin{proof}
		We first observe that we have a map $A \cot{C} A \to A \cot{D} A$ induced by:
		\[
			\begin{tikzcd}
				A \cot{C} A \ar[r, rightarrowtail] \ar[d, dashed] & A \otimes A \ar[r, shift left] \ar[r, shift right] \ar[d, equals] & A \otimes C \otimes A \ar[d, "A \otimes q \otimes A"]
				\\
				A \cot{D} A \ar[r, rightarrowtail] & A \otimes A \ar[r, shift left] \ar[r, shift right] & A \otimes D \otimes A
			\end{tikzcd}
		\]
		Next, since the left and right $D$-coactions on $A$ coincide, the following diagram commutes,
		\[
			\begin{tikzcd}[column sep = large]
				A \cot{D} A \ar[r, rightarrowtail] \ar[d] & A \otimes A \ar[d, "q_* \sigma"'] \ar[r, "q_* \sigma \otimes q_* \tau"] & D \otimes A \otimes D \otimes A \ar[d, "D \otimes \braid \otimes A"]
				\\
				D \otimes (A \cot{D} A) \ar[r, rightarrowtail] & D \otimes A \otimes A  \ar[r, "d \otimes A \otimes A"]& D \otimes D \otimes A \otimes A
		  	\end{tikzcd}
		\]	
		so the map $\begin{tikzcd} \iota\colon A \cot{C} A \ar[r] & A \cot{D} A \ar[r, rightarrowtail] & A \otimes A \end{tikzcd}$ is a comodule map over $d \colon D \to D \otimes D$.
		The comultiplication $\delta \colon B \to B \otimes B$ is also a comodule map over $d$, which we may combine with the above map to obtain a map $B \cot{D} (A \cot{C} A) \to A \otimes A$:
		\[
			\begin{tikzcd}
				B \cot{D} (A \cot{C} A) \ar[r, "\delta \cot{D} \iota"] \ar[dd, "?"', dashed] & (B \otimes B) \cot{D \otimes D} (A \otimes A) \ar[d, "\cong"]
				\\
				& (B \cot{D} A) \otimes (B \cot{D} A) \ar[d, "\varphi \otimes \varphi"]
				\\
				A \cot{C} A \ar[r, rightarrowtail] & A \otimes A
			\end{tikzcd}
		\]
		Finally, a routine diagram chase, repeatedly invoking the naturality of $\otimes$, allows us to verify that this map does indeed factor through $A \cot{C} A$, giving the desired map.
	\end{proof}

	\begin{remark}
		The prestacks we have defined should techincally be called \emph{split} prestacks.
		However, as these are the only prestacks we consider in this paper, we omit the word `split'.
	\end{remark}

\section{Smash products} \label{sec:smash-products}
	Let $\A = (C,A)$ be a prestack over $\B = (D, B)$, with actions $f$ and $\varphi$ as above.
	
	We make $B \cot{D} C$ an object of $_C \Comod_C$, with left coaction induced by 
	the comonoid map $f \colon B \cot{D} C \to C$, and right coaction induced by the comonoid map $t \colon B \to D$,
	\[
		\begin{tikzcd}[row sep = small, column sep = normal]
			f_* \Delta \colon B \cot{D} C \ar[r] &[-15pt] (B \cot{D} C) \otimes (B \cot{D} C) \ar[r, "f \otimes (B \cot{D} C)"] &[30pt] C \otimes (B \cot{D} C) &[-20pt]
			\\
			t_* \Delta \colon B \cot{D} C \ar[r] & (B \cot{D} C) \otimes (B \cot{D} C) \ar[r, "(B \cot{D} C) \otimes (t \cot{D} C)"] & (B \cot{D} C) \otimes (D \cot{D} C) \ar[r, "\cong", phantom] & (B \cot{D} C) \otimes C
		\end{tikzcd}
	\]
	where $\Delta = \delta \cot{d_D} d_C$ is the comultiplication of $B \cot{D} C$.
	We also have a right $B$-coaction induced by the comonoid map $q \colon C \to D$:
	\[
		\begin{tikzcd}[row sep = small, column sep = normal]
			q_* \Delta \colon B \cot{D} C \ar[r] &[-15pt] (B \cot{D} C) \otimes (B \cot{D} C) \ar[r, "(B \cot{D} C) \otimes (B \cot{D} q)"] &[30pt] (B \cot{D} C) \otimes (B \cot{D} D) \ar[r, "\cong", phantom] &[-20pt] (B \cot{D} C) \otimes B
		\end{tikzcd}	
	\]

	\begin{lemma} \label{lem:maps-over-p}
		Let $\A$ be an internal prestack over $\B$.
		Then:
		\begin{enumerate}
			\item The coaction $\pi$ is a bicomodule map over $p$:
				\[
					\begin{tikzcd}[sep = large]
						C \ar[d, "p"'] & A \ar[d, "\pi"] \ar[l, dotted, "\sigma"'] \ar[r, dotted, "\tau"] & C \ar[d, "p"]
						\\
						C \otimes D & \ar[l, dotted, "\sigma \otimes d"'] A \otimes D \ar[r, dotted, "\tau \otimes d"] & C \otimes D
					\end{tikzcd}
				\]
			\item The coaction $q_* \Delta$ is a bicomodule map over $p$:
				\[
					\begin{tikzcd}[sep = large]
						C \ar[d, "p"'] & B \cot{D} C \ar[d, "q_*\Delta"] \ar[l, dotted, "f_*\Delta"'] \ar[r, dotted, "t_* \Delta"] & C \ar[d, "p"]
						\\
						C \otimes D & \ar[l, dotted, "f_* \Delta \otimes \sigma"'] (B \cot{D} C) \otimes B \ar[r, dotted, "t_*\Delta \otimes \tau"] & C \otimes D
					\end{tikzcd}
				\]						
			\item The coaction $f_* \Delta$ is a comodule map over $p$:
				\[
					\begin{tikzcd}[sep = large]
						C \ar[d, "p"'] & B \cot{D} C \ar[d, "f_*\Delta"] \ar[l, dotted, "f_*\Delta"'] 
						\\
						C \otimes D & \ar[l, dotted, "d \otimes \sigma"'] C \otimes (B \cot{D} C) 
					\end{tikzcd}
				\]	
		\end{enumerate}
	\end{lemma}
	\begin{proof}
		By \Cref{lem:coaction-over-cocomm}, the top squares of the following diagrams commute:
		\[
			\begin{tikzcd}[sep = large]
				C \ar[d, "d"'] & \ar[l, dotted, "\sigma"'] A \ar[d, "\sigma"] 
				\\
				C \otimes C \ar[d,"q \otimes C"'] & \ar[l, dotted, "d \otimes \sigma"'] C \otimes A \ar[d, "q \otimes A"] 
				\\
				D \otimes C \ar[d, "\braid"'] & \ar[l, dotted, "d \otimes \sigma"']  D \otimes A \ar[d, "\braid"]
				\\
				C \otimes D & \ar[l, dotted, "\sigma \otimes d"'] A \otimes D
			\end{tikzcd}
			\quad \quad \quad \quad	
			\begin{tikzcd}[sep = large]
				A \ar[d, "\tau"']  \ar[r, dotted, "\tau"] & C \ar[d, "d"]
				\\
				A \otimes C \ar[r, dotted, "\tau \otimes d"] \ar[d, "A \otimes q"'] & C \otimes C \ar[d, "C \otimes q"]
				\\
				A \otimes D \ar[r, dotted, "\tau \otimes d"] & C \otimes D				
			\end{tikzcd}	
		\]
		The remaining squares obviously commute.
		For the left square, since $C$ is cocommutative, the left vertical composite is $p$.
		The right vertical composite is $\pi$ because $q_* \sigma = \braid \pi$ by \Cref{lem:B-D-comod}.
		This proves the first item.

		For the second item, the left square commutes because
		\begin{align*}
			(p \otimes q_* \Delta) \circ (f_* \Delta) 
				&= (C \otimes D \otimes q_* \Delta) \circ (p \otimes B \cot{D} C) \circ (f \otimes B \cot{D} C) \circ \Delta 
			\\ {}^{f \text{ is a map over } D}
				&= (C \otimes D \otimes q_* \Delta) \circ (\braid \otimes B \cot{D} C) \circ (D \otimes f \otimes B \cot{D} C) \circ (\sigma \otimes B \cot{D} C) \circ \Delta
			\\ {}
				&= \bigg( \big(\braid \circ (D \otimes f) \circ \sigma \big) \otimes q_*\Delta \bigg) \circ \Delta
			\\ {}
				&= \bigg( \big(\braid \circ (D \otimes f) \circ \sigma \big) \otimes \left( \big(B \cot{D} C \otimes B \cot{D} q \big)  \circ \Delta \right) \bigg) \circ \Delta
			\\ {}^{\text{associativity of } \Delta}
				&= \bigg( \left(\big( \big(\braid \circ (D \otimes f) \circ \sigma\big) \otimes B \cot{D}C \big)  \circ \Delta \right) \otimes B\cot{D}q \bigg) \circ \Delta
			\\ {}
				&= \bigg( \big(\big((\braid \circ (D \otimes f)) \otimes B \cot{D} C\big) \circ \big(\sigma \otimes B \cot{D} C\big)   \circ \Delta \big) \otimes B\cot{D}q \bigg) \circ \Delta
			\\ {}^\text{\Cref{cor:B-delta-same-source}}
				&= \bigg( \big(\big((\braid \circ (D \otimes f)) \otimes B \cot{D} C\big) \circ \big(\braid \otimes B \cot{D} C\big) \circ \big(B \cot{D} C \otimes \sigma\big)  \circ \Delta \big) \otimes B\cot{D}q \bigg) \circ \Delta
			\\ {}
				&= \bigg( \big( (f \otimes \sigma)  \circ \Delta \big) \otimes B\cot{D}q \bigg) \circ \Delta
			\\ {}
				&= (f_* \Delta \otimes \sigma) \circ q_* \Delta.
		\end{align*}

		The right square of the second item and the square in the third item commute by similar arguments. \todo{Draw diagram}
	\end{proof}		

	We are now in a position to define smash products of internal prestack.

	\begin{theorem} \label{thm:Gr-is-cat}
		Let $(f, \varphi) \colon \A \to \A \otimes \B$ be an internal prestack.
		There is an internal category
		\[
			\A \rtimes \B := \big(C, A \cot{C} (B \cot{D} C) \big),
		\]
		which we call the \fm{smash product} of $\A$ with $\B$.
		Further, $\A \rtimes \B$ has the structure of a $\B$-comodule category.
	\end{theorem}
	\begin{proof}
		By \Cref{lem:maps-over-p}, $\pi, q_* \Delta$ and $f_*\Delta$ are all maps over $p$, allowing us to define the composite in \Cref{eq:comp-A-rtimes-B}.
		\begin{figure}
			\[
				\begin{tikzcd}[row sep = huge]
					A \cot{C} (B \cot{D} C)
					\cot{C} 
					A \cot{C} (B \cot{D} C)
					\ar[d, "\pi \,\cot{p}\, (q_* \Delta) \,\cot{p}\, \pi \,\cot{p}\, (f_* \Delta)" ]
					\\
					\bigg(A \otimes D \bigg) 
					\cot{C \otimes D} 
					\bigg( (B \cot{D} C) \otimes B \bigg) 
					\cot{C \otimes D} 
					\bigg(A \otimes D\bigg) 
					\cot{C \otimes D} 
					\bigg( C \otimes (B \cot{D} C) \bigg)
					\ar[d, "\cong"]
					\\
					\bigg(A \cot{C} (B \cot{D} C) \cot{C} A \cot{C} C \bigg) 
					\otimes 
					\bigg(D \cot{D} B \cot{D} D \cot{D} (B \cot{D} C) \bigg)
					\ar[d, "\cong"]
					\\
					\bigg( A \cot{C} B \cot{D} A \bigg)
					\otimes
					\bigg( B \cot{D} B \cot{D} C \bigg) \ar[d, "(A \,\cot{C}\, \varphi) \otimes (m \,\cot{D}\, C) " ]
					\\
					\bigg( A \cot{C} A \bigg) \otimes \bigg( B \cot{D} C \bigg) \ar[d, "m \otimes (B \cot{D} C)" ]
					\\
					A \otimes (B \cot{D} C)
				\end{tikzcd}
			\]
			\caption{Composition in the internal category $\A \rtimes \B$}
			\label{eq:comp-A-rtimes-B}		
		\end{figure}			
		In fact, this composite factors through $A \cot{C} (B \cot{D} C)$, giving the multiplication on $\A \rtimes \B$. \todo{Verify}

		The unit of $\A \rtimes \B$ is given by the composite:
		\[
			\begin{tikzcd}[column sep = large]
				C \ar[r, "\cong"] & C \cot{C} (D \cot{D} C) \ar[r, "u \cot{C} ( u \cot{D} C)"] & A \cot{C} (B \cot{D} C)
			\end{tikzcd}
		\]
		These maps are unital and associative (because of \Cref{def:prestack-functorial-in-B,def:prestack-functorial-in-A} in \Cref{def:prestack}, and the fact that $\A$ and $\B$ are internal categories), so $\A \rtimes \B$ is an internal category.

		To see that $\A \rtimes \B$ has the structure of a $\B$-comodule category, note that $C$ already has a $D$-coaction $p$.
		We then take the $B$-coaction on $A \cot{C} (B \cot{D} C)$ to be the composite in \Cref{eq:B-coaction-A-rtimes-B}.
		\begin{figure}
			\[
				\begin{tikzcd}[row sep = large]
					A \cot{C} (B \cot{D} C)
					\ar[d, "\pi \,\cot{C}\, (q_*\Delta)" ]
					\\
					\bigg(A \otimes D \bigg) \cot{C \otimes D} \bigg((B \cot{D} C) \otimes B \bigg)
					\ar[d, "\cong"]
					\\
					\bigg( A \cot{C} (B \cot{D} C) \bigg)  \otimes \bigg(D \cot{D} B \bigg)
					\ar[d, "\cong"]
					\\
					\bigg( A \cot{C} (B \cot{D} C) \bigg)  \otimes B
				\end{tikzcd}
			\]	
			\caption{$\B$-coaction on $\A \rtimes \B$}
			\label{eq:B-coaction-A-rtimes-B}		
		\end{figure}					
	\end{proof}

\section{Coinvariants of comodule categories} \label{sec:coinvariants}
	Although we have not defined what a `cartesian fibered right $\B$-comodule category' should be, we can still verify that the fibers of $\A \rtimes \B$ allow us to recover our original prestack $\A$.
	We first begin by defining the fibers of any right $\B$-comodule category.

	\begin{definition}
		Let $\A$ be a right $\B$-comodule category, with coaction functor $p \colon \A \to \A \otimes \B$. 
		The \fm{coinvariant category} is the coinduction $\A$ along $(D, u)\colon \D \to \B$:
		\[
			\begin{tikzcd}[sep = large]
				\A\cot{\B} \D \ar[r] \ar[d, dotted] \ar[dr, very near start, phantom, "\lrcorner"] & \A \ar[d, dotted]
				\\
				\D \ar[r, "{(D, u)}"] & \B
			\end{tikzcd}
		\]	
		Equivalently, $\A \cot{\B} \D$ is given by the equalizer:
		\[
			\begin{tikzcd}[column sep = huge]
				\A \cot{\B} \D \ar[r, rightarrowtail]
				&[-30pt]
				\A \otimes \D \ar[r, shift left, "{(p, \pi) \otimes \D}"] 
				\ar[r, shift right, "\A \otimes \big(({D,u}) \otimes \D\big)  ({d,\delta})"']
				&[+30pt]
				\A \otimes \B \otimes \D
			\end{tikzcd}
		\]	
	\end{definition}	
	The following lemmas follow almost by definition:
	\begin{lemma}
		The coinvariant category  $\A \cot{\B} \D$ is a $\D$-comodule category.
	\end{lemma}

	\begin{lemma}
		The coinvariant category is given by $\A \cot{\B} \D = (C \cong C \cot{D} D,\, A \cot{B} D)$.
	\end{lemma}

	Given an arbitrary $\B$-comodule category, it is unlikely that its coinvariant category has the structure of a prestack over $\D$.
	However, when the $\B$-comodule category is of the form $\A \rtimes \B$ for a prestack $\A$, we have:

	\begin{theorem} \label{thm:internal-coinvariant}
		Let $\A$ be a prestack over $\B$ and let $\A \rtimes \B$ be the corresponding right $\B$-comodule category.
		Then the coinvariant category $(\A \rtimes \B) \cot{\B} \D$ is a prestack over $\B$, which is moreover isomorphic to $\A$.
	\end{theorem}
	\begin{proof}
		First observe that the comonoid of objects for $\A, \A\rtimes \B$ and $(\A \rtimes \B) \cot{\B} \D$ are all $C$.
		On morphisms, recall that the $B$-coaction on $A \cot{C} (B \cot{D} C)$ is given by the copy of $B$ sitting inside $B \cot{D} C$.
		Thus \todo{Draw diagram for first isomorphism}
		\begin{align*}
			A \cot{C} (B \cot{D} C) \cot{B} D &\cong A \cot{C} (D \cot{D} C) \\
											  &\cong A \cot{C} C \\
											  &\cong A.
		\end{align*}
		So $\A$ and $(\A \rtimes \B) \cot{\B} \D$ are isomorphic categories.
		We may then transfer the prestack structure of $\A$ over to $(\A \rtimes \B) \cot{\B} \D$.
	\end{proof}

\section{Further work} \label{sec:further-work}
	In this paper, we have seen that the smash product (a.k.a. the Grothendieck construction) for split prestacks $F \colon B^\op \to \Cat$ generalizes well to the non-cartesian internal setting, as long as one is willing to relax the definition of what it means to be a prestack.

	Several assumptions were made that reduce the scope of the results in this paper.
	Firstly, although $\V$ is not cartesian, it is assumed to be \emph{symmetric} monoidal rather than merely braided monoidal. 
	The symmetry assumption yielded certain convenient but not crucial lemmas (at least in the author's opinion).
	It is thus believable that smash products as defined in this paper should still exist in the braided monoidal setting.

	We have also assumed that the comonoids of objects of both the base $\B$ and the prestack $\A$ are cocommutative.
	This assumption seems to be more crucial, and not something that can be easily done away with.
	Indeed, \Cref{lem:maps-over-p} -- the main technical result -- holds only because we assumed that $C$ and $D$ were cocommutative.

	It remains to be seen if a similar construction can be carried out for internal categories with non-cocommutative comonoids of objects, e.g.\ the quantum categories of \cite{chikhladze2011category}.
	We note that smash products have been defined for weak bialgebras \cite{nikshych2000duality}, and that these are bimonoids in an appropriate duoidal category, so a possible next step would be to define smash products for prestacks internal to a duoidal category.


\bibliography{../chapters/Biblio}
\end{document}